\pdfoutput=1
\RequirePackage{ifpdf}
\ifpdf 
\documentclass[pdftex]{sigma}
\else
\documentclass{sigma}
\fi

\usepackage{stmaryrd, enumerate}

\newcommand*{\un}{{\mathbf 1}}
\newcommand{\tr}{\triangleright}
\newcommand{\BCH}{\operatorname{BCH}}

\numberwithin{equation}{section}

\newtheorem{Theorem}{Theorem}[section]
\newtheorem{Corollary}[Theorem]{Corollary}

\newtheorem{Proposition}[Theorem]{Proposition}
 { \theoremstyle{definition}
\newtheorem{Definition}[Theorem]{Definition}

\newtheorem{Remark}[Theorem]{Remark} }

\begin{document}

\allowdisplaybreaks

\newcommand{\arXivNumber}{1505.02436}

\renewcommand{\PaperNumber}{093}

\FirstPageHeading

\ShortArticleName{Post-Lie Algebras and Isospectral Flows}

\ArticleName{Post-Lie Algebras and Isospectral Flows}

\Author{Kurusch {EBRAHIMI-FARD}~$^{\dag^1}$, Alexander {LUNDERVOLD}~$^{\dag^2}$, Igor {MENCATTINI}~$^{\dag^3}$\\ and Hans Z.~{MUNTHE-KAAS}~$^{\dag^4}$}

\AuthorNameForHeading{K.~Ebrahimi-Fard, A.~Lundervold, I.~Mencattini and H.Z.~Munthe-Kaas}

\Address{$^{\dag^1}$~ICMAT, C/ Nicol\'as Cabrera 13-15, 28049 Madrid, Spain} 
\EmailDD{\href{mailto:kurusch@icmat.es}{kurusch@icmat.es}}
\URLaddressDD{\url{http://www.icmat.es/kurusch/personal}}

\Address{$^{\dag^2}$~Department of Computing, Mathematics and Physics, Faculty of Engineering,\\
 \hphantom{$^{\dag^2}$}~Bergen University College, Postbox 7030, N-5020 Bergen, Norway}
\EmailDD{\href{mailto:alexander.lundervold@gmail.com}{alexander.lundervold@gmail.com}}
\URLaddressDD{\url{http://alexander.lundervold.com}}

\Address{$^{\dag^3}$~Instituto de Ci\^encias Matem\'aticas e de Computa\c{c}\~ao, Universidade de S\~ao Paulo,\\
 \hphantom{$^{\dag^3}$}~Campus de S\~ao Carlos, Caixa Postal 668, 13560-970 S\~ao Carlos, SP, Brazil}
\EmailDD{\href{mailto:igorre@icmc.usp.br}{igorre@icmc.usp.br}}

\Address{$^{\dag^4}$~Department of Mathematics, University of Bergen, Postbox 7803, N-5020 Bergen, Norway}
\EmailDD{\href{mailto:hans.munthe-kaas@math.uib.no}{hans.munthe-kaas@math.uib.no}}
\URLaddressDD{\url{http://hans.munthe-kaas.no}}

\ArticleDates{Received August 13, 2015, in f\/inal form November 16, 2015; Published online November 20, 2015}

\Abstract{In this paper we explore the Lie enveloping algebra of a post-Lie algebra derived from a classical $R$-matrix. An explicit exponential solution of the corresponding Lie bracket f\/low is presented. It is based on the solution of a post-Lie Magnus-type dif\/ferential equation.}


\Classification{70H06; 17D99; 37J35}

\section{Introduction}

\emph{Isospectral flows} and the corresponding Lax type equations play an important role in the theo\-ry of dynamical systems, both in f\/inite and inf\/inite dimensions. See \cite{BaBeTa,Lax}. They appear together with a large supply of conserved quantities for the original dynamical system. In the f\/inite-dimensional case, i.e., for systems with a f\/inite number of degrees of freedom, the Lax representation may correspond to the Hamiltonian representation of the dynamical system in terms of Euler-type equations on the coadjoint orbits of a suitable Lie group $G$. Writing $\mathfrak g^*$ for the dual of the Lie algebra $\mathfrak g$ corresponding to $G$, recall that if $H \in C^{\infty}(\mathfrak g^*)$ is a Hamiltonian of the dynamical system then the corresponding Hamiltonian equations, written with respect to the canonical linear Poisson structure $\{\cdot,\cdot\}_{\mathfrak g}$, take the following form
\begin{gather}
\label{eq:uno}
	\dot{\alpha} = -\operatorname{ad}^{\sharp}_{dH_{\alpha}}(\alpha),
\end{gather}
where $\alpha\in\mathfrak g^\ast$.
In this description the Casimir functions provide only trivial f\/irst integrals with respect to the bracket $\{\cdot,\cdot\}_{\mathfrak g}$. The existence of a map $R \in \operatorname{End}(\mathfrak g)$, which satisf\/ies the so-called modif\/ied classical Yang--Baxter equation \cite{STS}
\begin{gather}
\label{eq:mcybe}
	[R(x),R(y)] - R\big([R(x),y]+[x,R(y)]\big) = -[x,y],
\end{gather}
usually called an $R$-matrix, yields a \emph{new} Lie bracket on $\mathfrak g$
\begin{gather}
\label{eq:double}
	[x,y]_R := \frac{1}{2}\big([R(x),y] + [x,R(y)]\big).
\end{gather}
This Lie bracket def\/ines the double Lie algebra $\bar{\mathfrak{g}}$, and the corresponding linear Poisson structu\-re~$\{\cdot,\cdot\}_R$ on~$\mathfrak g^*$ associates to each Casimir function with respect to $\{\cdot,\cdot\}_{\mathfrak g}$, a non-trivial f\/irst integral of the original dynamical system. In this augmented setting the following holds:
\begin{enumerate}\itemsep=0pt
\item[(i)] The Casimir functions corresponding to the \emph{original} Poisson structure $\{\cdot,\cdot\}_{\mathfrak g}$ Poisson commute with respect to the new Poisson structure $\{\cdot,\cdot\}_R$.

\item[(ii)] For every Casimir function $H$, equation \eqref{eq:uno} written with respect to the Poisson struc\-tu\-re~$\{\cdot,\cdot\}_R$ assumes the form
\begin{gather}
\label{eq:quattro}
	\dot{\alpha} = -\frac{1}{2}\operatorname{ad}^{\sharp}_{RdH_{\alpha}}(\alpha),\qquad
	\forall \, \alpha \in \mathfrak g^{\ast}.
\end{gather}
Furthermore, under the assumption that $\mathfrak g$ is endowed with a non-degenerate, $\mathfrak g$-invariant bilinear form $(\cdot \vert \cdot)$, the previous equation can be written as the following Lie bracket f\/low equation
\begin{gather}
\label{eq:cinque}
	\dot{x}_{\alpha} = \frac{1}{2}[x_{\alpha},RdH_{\alpha}],
\end{gather}
where $x_{\alpha} \in \mathfrak g$ is the (unique) element such that $(x_{\alpha} \vert y) = \langle \alpha,y \rangle$, for all $y \in\mathfrak g$.
\end{enumerate}

In many interesting cases \cite{BaBeTa,ChuNorris,Faybusovich,Suris,Watkins}, the existence of such an $R$-matrix turns out to be equivalent to a decomposition of the Lie algebra $\mathfrak g = \mathfrak g_+ \oplus \mathfrak g_-$, where $\mathfrak g_\pm$ are two Lie subalgebras of $\mathfrak g$. To any such decomposition corresponds a (local) decomposition of the corresponding Lie group $G \simeq G_+ \times G_-$, where $\simeq$ in this case means a local dif\/feomorphism from a neighborhood of the identity $e \in G$ to a neighborhood of the identity $(e,e)$ of the (product) Lie group $G_+ \times G_-$. Regarding equation~\eqref{eq:quattro}, or equivalently equation~\eqref{eq:cinque}, the following factorization theorem holds true. See the above references for details and background.

\begin{Theorem} [\cite{STS}] \label{thm:factorization}
Let $\alpha_0 \in \mathfrak g^{\ast}$, let $H$ be a Casimir function of $(\mathfrak g^{\ast}, \{\cdot,\cdot\}_{\mathfrak g})$. Let $g_{\pm}(t)$ be two smooth curves in $G$, such that:
\begin{enumerate}\itemsep=0pt
\item[$(a)$] $g_\pm(t)\in G_\pm$ for all $t$ for which they are defined,
\item[$(b)$] $g_{\pm}(0)$ are both equal to the identity $e\in G$,
\item[$(c)$] they give a unique solution of the following factorization problem
\begin{gather}
\label{thm1:fact}
	\exp (tdH_{\alpha_0})=g_+(t)g_-(t),
\end{gather}
at least for $\vert t \vert < \epsilon$, with $\epsilon > 0$.
\end{enumerate}
Then the curve $t \rightsquigarrow \alpha(t)$, $\alpha(0)=\alpha_0$, given by
\begin{gather*}
	\alpha(t)= \operatorname{Ad}^{\sharp}_{g^{-1}_{+}(t)}\alpha_0
		     = \operatorname{Ad}^{\sharp}_{g_{-}(t)}\alpha_0
\end{gather*}
is a solution of equation~\eqref{eq:quattro}.
\end{Theorem}

 The previous theorem connects a certain factorization of elements in the Lie group $G$ with the solution of Lie bracket f\/low equations in the corresponding (dual, $\mathfrak g^*$, of the) Lie algebra~$\mathfrak{g}$. The main aim of this work is to explore this result in the framework of the Lie enveloping algebra of a post-Lie algebra def\/ined on~$\mathfrak{g}$ in terms of an $R$-matrix.

A \emph{post-Lie algebra} \cite{Burde,EFLMK,LMK,Vallette} consists of a vector space $V$ equipped with two Lie brac\-kets~$[ \cdot , \cdot ]$ and $\llbracket \cdot,\cdot \rrbracket$ as well as a non-commutative and non-associative product $\tr\colon  V \otimes V \to V$, such that the following identity holds
\begin{gather}
\label{def:PostLieRel}
	 \llbracket x,y \rrbracket = x \tr y - y \tr x + [x,y].
\end{gather}
Further  below we will state the precise def\/inition and relations that characterize such an algebraic structure. In light of~\eqref{eq:mcybe} and~\eqref{eq:double}, the third product on the Lie algebra $(\mathfrak{g}, [ \cdot , \cdot ])$ is given in terms of the $R$-matrix map~\cite{GuoBaiNi}, $x \tr y := [\frac{1}{2}(R + \operatorname{id})(x),y]$, $x,y \in \mathfrak g$. Lifting the post-Lie algebra to the Lie enveloping algebra of the Lie algebra $(\mathfrak{g}, [ \cdot , \cdot ])$ allows us to def\/ine another associative product on $\mathcal{U}(\mathfrak{g})$, which is compatible with the latter's coalgebra structure \cite{EFLMK}. The resulting Hopf algebra is isomorphic -- as a Hopf algebra -- to the Lie enveloping algebra of the second Lie algebra $(\bar{\mathfrak{g}}, \llbracket \cdot , \cdot \rrbracket)$. As a result $\mathcal{U}(\mathfrak{g})$ is equipped with two natural exponential maps, and the relation between those and the corresponding Lie groups in $\mathcal{U}(\mathfrak{g})$ is captured through a~Magnus-type dif\/ferential equation. This gives rise to explicit solutions of the factorization~\eqref{thm1:fact} in Theorem~\ref{thm:factorization}.

\looseness=1
We close this introduction with two remarks. First we would like to mention that dif\/ferential geometry is a natural place to look for examples of post-Lie algebras. Indeed, a Koszul connection yields a $\mathbb{R}$-bilinear product on the space of smooth vector f\/ields $\mathcal{X}(\mathcal{M})$ on a mani\-fold~$\mathcal{M}$. Flatness and constant torsion together with the Bianchi identities imply relation~\eqref{def:PostLieRel} between the Jacobi--Lie bracket of vector f\/ields, the torsion itself, and the product def\/ined in terms of the connection~\cite{LMK}.
Second, we would like to stress that the formalism introduced in this note is based on the theory of classical $R$-matrices, and will be applied only in the context of classical dynamical systems, whose descriptions are given in terms of isospectral f\/lows. However, saying this, it is worth mentioning that a similar formalism was used in~\cite{STS1} and~\cite{RSTS} to the study of quantum groups and quantum integrable systems, see also Remark~\ref{rem:im} \mbox{below}.

\emph{Outline of the paper.} Sections~\ref{sect:R-Matrices} and \ref{sect:RmatLieAdm} contain several preliminary results, which are crucial for the main statements of this paper, to be found in Sections~\ref{sect:LieEnveloping} and~\ref{sect:PostLieBracketFlow}.
 More precisely, in Section~\ref{sect:R-Matrices} we quickly recall the basic notions of $R$-matrices and their relations to factorization problems already mentioned above. Section~\ref{sect:RmatLieAdm} collects some basic facts about Lie-admissible algebras and post-Lie algebras. In particular, it is shown that solutions of the modif\/ied classical Yang--Baxter equation yield a post-Lie algebra structure on the original Lie algebra~$\mathfrak{g}$.
 In Section~\ref{sect:LieEnveloping}, after recalling some important result on the universal enveloping algebra of a post-Lie algebra, we state a factorization theorem for the generators of a group $G^{\ast}$ sitting inside the universal enveloping algebra of any Lie algebra endowed with a solution of the modif\/ied classical Yang--Baxter equation. Furthermore, a distinguished linear isomorphism is def\/ined between the universal enveloping algebra $\mathcal U(\mathfrak g)$ of a Lie algebra $\mathfrak g$ supporting an $R$-matrix, and the universal enveloping algebra~$\mathcal U(\bar{\mathfrak g})$ of the corresponding double Lie algebra. Moreover, it is shown that by swapping the associative product of~$\mathcal U(\mathfrak g)$ for a~new product def\/ined by extending the post-Lie product def\/ined on~$\mathfrak g$ in terms of the $R$-matrix to $\mathcal U(\mathfrak g)$, such a~linear isomorphism becomes a~morphism of associative algebras. Finally, in Section~\ref{sect:PostLieBracketFlow} the post-Lie algebra structure is invoked to show how the $\BCH$-recursion follows as the solution of a~Magnus-type dif\/ferential equation. This is then applied to Lie bracket f\/lows. It is shown that the solution of a Lie bracket f\/low on a Lie algebra $\mathfrak g$ endowed with a solution of the modif\/ied classical Yang--Baxter equation can be described, under suitable convergence assumptions, in terms of the generators of the group~$G^{\ast}$.

In the following the ground f\/ield~$\mathbb{K}$ is of characteristic zero, and~$\mathbb{K}$-algebras are assumed to be associative and unital, if not stated otherwise.

\section[$R$-matrices and factorization]{$\boldsymbol{R}$-matrices and factorization}
\label{sect:R-Matrices}

We consider a Lie algebra $\mathfrak g$ together with an $R$-matrix $\pi_+\colon \mathfrak g \to \mathfrak g$. Both $\pi_+$ and the map $\pi_-:= \operatorname{id} - \pi_+$ satisfy the Lie algebra identity
\begin{gather}
\label{mYBE}
	[\pi_\pm(x),\pi_\pm(y)] + \pi_\pm([x,y]) = \pi_\pm([\pi_\pm(x),y] + [x,\pi_\pm(y)]).
\end{gather}
Note that the map $R := \operatorname{id} - 2 \pi_+$ satisf\/ies the modif\/ied classical Yang--Baxter equation \eqref{eq:mcybe}. For details we refer the reader to \cite{BaBeTa,Faybusovich,RSTS,STS}. In the following theorem we collect some well-known results.

\begin{Theorem}[\cite{STS}]
If $\pi_-$ is a solution of \eqref{mYBE}, then the bracket $\llbracket \cdot,\cdot \rrbracket$ defined for all $x,y \in\mathfrak g$~by
\begin{gather}
\label{doublebracket}
	\llbracket x,y \rrbracket := [\pi_-(x),y] + [x,\pi_-(y)] - [x,y] = [\pi_-(x),\pi_-(y)] - [\pi_+(x),\pi_+(y)],
\end{gather}
satisfies the Jacobi identity, and therefore defines the so-called double Lie algebra $\bar{\mathfrak{g}}$ on the vector space underlying $\mathfrak g$. Both images $\mathfrak g_\pm:=\pi_\pm(\mathfrak g)$ form Lie subalgebras, and $\pi_\mp (\llbracket x,y \rrbracket) = \pm [\pi_\mp(x),\pi_\mp(y)]$.
\end{Theorem}

As mentioned in the introduction, solutions of \eqref{mYBE} are intimately related to factorizations of the Lie group $G$, see \cite{Faybusovich} for details. For simplicity we assume that~$\pi_+$ is a~projector~-- which is covering many interesting cases. The subgroups corresponding to the Lie subalgebras $\mathfrak{g}_\pm$ are denoted~$G_\pm$. For $a_0 \in \mathfrak g$ and a small enough $t$, that is, in a suf\/f\/iciently small neighborhood of the unit of the group $G$ corresponding to $\mathfrak g$, the following unique factorization holds
\begin{gather}
\label{Fact1}
	\exp(t a_0) = g_+(t)g_-(t),
\end{gather}
with $g_\pm (t) \in G_\pm$ for all $t$ for which they are def\/ined. Def\/ine the map
\begin{gather*}
	a(t) := g^{-1}_+(t) a_0 g_+(t) = g_-(t) a_0 g^{-1}_-(t),
\end{gather*}
and recall that it satisf\/ies the Lie bracket initial value problem
\begin{gather*}
	\dot{a}(t) =  [a(t),p_+(t)],\qquad a(0)=a_0,
\end{gather*}
with $p_+(t) := g_+^{-1}(t) \frac{d}{dt} (g_+(t))$. From~\eqref{Fact1} it follows that
\begin{gather*}
	a(t) = g_+^{-1}(t) \frac{d}{dt} (g_+(t)) +  \frac{d}{dt} (g_-(t))g_-^{-1}(t).
\end{gather*}
Since $\pi_\pm$ are projectors we obtain $g_+^{-1}(t) \frac{d}{dt} (g_+(t)) \in \mathfrak{g}_+$ and $\frac{d}{dt} (g_-(t))g_-^{-1}(t) \in \mathfrak{g}_-$, and therefore $p_+(t)=\pi_+(a(t))$.

Anticipating what follows below, we remark that in \cite{CasasIserles} the function $\Omega(t;a_0)$ was described in terms of a Magnus-type dif\/ferential equation, such that $g_+(t)=\exp(\Omega(t;a_0))$. In this paper we will describe this Magnus-type dif\/ferential equation using a post-Lie algebra, and thereby clarify its link to~\eqref{Fact1} by showing that its solution is given in terms of the $\BCH$-recursion~\cite{EGM}.

\section[Lie-admissible algebras, post-Lie algebras and $R$-matrices]{Lie-admissible algebras, post-Lie algebras and $\boldsymbol{R}$-matrices}
\label{sect:RmatLieAdm}

\begin{Definition}
A $\mathbb K$-algebra $(A,\cdot)$, not necessarily associative, is called \emph{Lie-admissible} if the commutator $[a,b] := a \cdot b - b \cdot a$ def\/ines a Lie bracket. In this case, the corresponding Lie algebra $(A,[\cdot,\cdot])$ will be denoted by $A_{\text{Lie}}$.
\end{Definition}

Note that associative $\mathbb K$-algebras are Lie-admissible. Another class of Lie-admissible algebras is introduced in the following def\/inition.

\begin{Definition}
The algebra $(A,\diamond)$ with binary product $\diamond\colon A \otimes A \to A$ will be called a \emph{$($left$)$ pre-Lie algebra}, if for all $x,y,z \in A$
\begin{gather}
\label{preLie}
	 {\rm{a}}_{\diamond}(x,y,z)={\rm{a}}_{\diamond}(y,x,z),
\end{gather}
where ${\rm{a}}_{\diamond}(x,y,z) := x \diamond (y \diamond z) - (x \diamond y) \diamond z$ is the associator.
\end{Definition}

Pre-Lie algebras are Lie-admissible. Indeed, note that identity~\eqref{preLie} can be written as $\ell_{[x,y] \diamond}(z) = [\ell_{x \diamond},\ell_{y \diamond}](z)$, where the linear map $\ell_{x \diamond}\colon A \to A$ is def\/ined by $\ell_{x \diamond}(y):=x \diamond y$ and the bracket on the left-hand side is def\/ined by $[x,y] := x \diamond y - y \diamond x$. As a consequence it satisf\/ies the Jacobi identity, turning $A$ into a Lie algebra. See \cite{Cartier,Manchon} for more details. We now turn to the def\/inition of post-Lie algebra following reference~\cite{LMK}.

\begin{Definition} \label{def:postLie}
Let $(\mathfrak g, [\cdot,\cdot])$ be a Lie algebra, and let $\triangleright  \colon {\mathfrak g} \otimes {\mathfrak g} \rightarrow \mathfrak g$ be a binary product such that, for all $x,y,z \in \mathfrak g$,
\begin{gather*}
	x \triangleright [y,z] = [x\triangleright y , z] + [y , x \triangleright z],
\end{gather*}
and
\begin{gather*}
	[x,y] \triangleright z = {\rm{a}}_{\triangleright  }(x,y,z) - {\rm{a}}_{\triangleright  }(y,x,z).
\end{gather*}
Then the triplet $(\mathfrak{g}, [\cdot,\cdot], \triangleright)$ is called a \emph{post-Lie algebra}.
\end{Definition}

\begin{Remark}\label{rmk:postlie-prod}
Let $(\mathfrak{g}, [\cdot,\cdot], \triangleright)$ be a post-Lie algebra.
\begin{enumerate}\itemsep=0pt
	\item[a)] If $x \blacktriangleright y := x \triangleright y + [x,y]$, then $(\mathfrak g, -[\cdot,\cdot], \blacktriangleright)$
	 	is a post-Lie algebra.	\label{b}  \smallskip
	\item[b)] If $x \succ y := x \triangleright y + \frac{1}{2}[x,y]$, then $(\mathfrak g,\succ)$ is
		Lie-admissible.		\label{c}
\end{enumerate}
\end{Remark}

\begin{Proposition}
\label{prop:post-lie}
Let $(\mathfrak g, [\cdot, \cdot], \tr)$ be a post-Lie algebra. The bracket
\begin{gather}
\label{postLie3}
	\llbracket x,y \rrbracket := x \tr y - y \tr x + [x,y]
\end{gather}
satisfies the Jacobi identity for all $x, y \in \mathfrak g$. The Lie algebra is written as $(\bar{\mathfrak g},\llbracket \cdot,\cdot \rrbracket)$.
\end{Proposition}

\begin{Remark}
Pre- and post-Lie algebras are important in the theory of numerical methods for dif\/ferential equations. We refer the reader to \cite{Cartier, ChaLiv, EFLMK, Manchon,LMK} for background and details.
\end{Remark}

In the introduction we pointed at an archetypal example of post-Lie algebra coming from dif\/ferential geometry. Here we will present an algebraic example using $R$-matrices. Let $\pi_+ \in \operatorname{End}(\mathfrak g)$ satisfy identity~\eqref{mYBE}, and def\/ine the following binary product on $\mathfrak g$:
\begin{gather}
\label{def:RBpostLie}
	a \triangleright b := - [\pi_+(a),b].
\end{gather}

\begin{Theorem}[\cite{GuoBaiNi}]
\label{thm:postLie1}
The product \eqref{def:RBpostLie} defines a post-Lie algebra structure on~$\mathfrak g$.
\end{Theorem}

It turns out that the new Lie bracket def\/ined in terms of this post-Lie product and the original Lie bracket is the one given in~\eqref{doublebracket}. Indeed, for $\pi_+:= \operatorname{id} - \pi_-$,
\begin{gather*}
	x \triangleright y - y \triangleright x + [x,y] = [\pi_-(x),y] + [x,\pi_-(y)] - [x,y] = \llbracket x,y \rrbracket.
\end{gather*}

\begin{Remark}
Returning to item a) in Remark \ref{rmk:postlie-prod}, note that the second post-Lie algebra $(\mathfrak g, -[\cdot,\cdot], \blacktriangleright)$ is def\/ined in terms of $\pi_-:=\operatorname{id}- \pi_+$: the product is $x \blacktriangleright y := x \triangleright y + [x,y] = [\pi_-(a),b]$.

As for item b), one f\/inds that the Lie-admissible algebra $(\mathfrak g,\succ)$ is def\/ined through the binary composition $x \succ y := x \triangleright y + \frac{1}{2}[x,y] = [\frac{1}{2}R(x),y]$, where $R:= \operatorname{id} - 2 \pi_-$. The Lie bracket \eqref{postLie3} is then given by $\llbracket x,y \rrbracket := x \succ y - y \succ x,$ which is just another way of writing \eqref{eq:double}. With $\tilde{R}:=\frac{1}{2}R$, one can deduce from $[\tilde{R}(x),\tilde{R}(y)] - \tilde{R}\big([\tilde{R}(x),y]+[x,\tilde{R}(y)]\big) = -\frac{1}{4}[x,y]$ that
\begin{gather*}
	 {\rm{a}}_{\succ}(x,y,z) - {\rm{a}}_{\succ}(y,x,z) = -\frac{1}{4}[[x,y],z].
\end{gather*}
\end{Remark}

\section{Lie enveloping algebra of a post-Lie algebra}
\label{sect:LieEnveloping}

In \cite{EFLMK} the Lie enveloping algebra of a post-Lie algebra was described. Here we recall the basic results without proofs. Let $(\mathfrak g, [\cdot, \cdot], \tr)$ be a post-Lie algebra, and $\mathcal{U}(\mathfrak g)$ the universal enveloping algebra of the Lie algebra  $(\mathfrak g, [\cdot, \cdot])$. Recall that $\mathcal{U}(\mathfrak g)$ with concatenation product is a non-commutative, cocommutative f\/iltered Hopf algebra generated by~$\mathfrak g \hookrightarrow \mathcal{U}(\mathfrak g)$. The coshuf\/f\/le coproduct is def\/ined for $x \in \mathfrak g$ by $\Delta(x) := x \otimes \un + \un \otimes x$, i.e., elements of $\mathfrak g$ are primitive. It is extended multiplicatively to all of $\mathcal{U}(\mathfrak g)$. We use Sweedler's notation for the coproduct: $\Delta(T)=:T_{(1)} \otimes T_{(2)}$, see~\cite{Sw}. The counit is denoted by $\epsilon\colon \mathcal{U}(\mathfrak g) \to \mathbb{K}$, and the antipode $S\colon \mathcal{U}(\mathfrak g)\rightarrow \mathcal{U}(\mathfrak g)$ is def\/ined through $S(x_1\cdots x_k):=(-1)^k x_k\cdots x_1$ (in particular $S(x)=-x$ for all $x\in\mathfrak g$). Finally, remember that the universal property of $\mathcal{U}(\mathfrak g)$ implies that if $A$ is an associative algebra and $f\colon {\mathfrak g}\rightarrow A_{\text{Lie}}$ is a homomorphism of Lie algebras, then there exists a unique (unital associative algebra) morphism $F\colon \mathcal{U}(\mathfrak g) \rightarrow A$ such that  $F\circ i=f$, where $i\colon \mathfrak g \rightarrow \mathcal{U}(\mathfrak g)$ is the canonical embedding.

We have seen in Proposition~\ref{prop:post-lie} that the vector space underlying a post-Lie algebra carries two Lie algebras, $(\mathfrak g, [\cdot, \cdot])$ and $(\bar{\mathfrak g}, \llbracket \cdot , \cdot \rrbracket )$, related via the post-Lie product
\begin{gather*}
	\llbracket x,y \rrbracket = x \tr y - y \tr x + [x,y].
\end{gather*}

In what follows, $(\mathcal{U}(\bar{\mathfrak g}),\cdot)$ will denote the universal enveloping algebra of the Lie algebra $(\bar{\mathfrak g}, \llbracket \cdot , \cdot \rrbracket )$. In the next proposition the post-Lie product is extended to~$\mathcal{U}(\mathfrak g)$.

\begin{Theorem}[\cite{EFLMK}]\label{thm:OudomGuin}
There is a unique extension of the post-Lie product $\tr$ from $\mathfrak g$ to $\mathcal{U}(\mathfrak g)$. On $(\mathcal{U}(\mathfrak g),\tr)$ the product
\begin{gather}
\label{def:preLie}
	A * B := A_{(1)}(A_{(2)} \tr B)
\end{gather}
for $A,B \in \mathcal{U}(\mathfrak g)$ is associative and unital. Moreover, $(\mathcal{U}(\mathfrak g), *, \Delta)$ is a Hopf algebra isomorphic to $(\mathcal{U}(\bar{\mathfrak g}),\cdot,\Delta)$.
\end{Theorem}

\begin{Remark}\label{rmk:rem} A few remarks are in order.
\begin{enumerate}[i)]\itemsep=0pt

\item\label{rmk:0} The last statement in Theorem \ref{thm:OudomGuin} appeared in the context of pre-Lie algebras in \cite{OudomGuin}, and we refer the reader to \cite{EFLMK} for details. For the notation, see item (\ref{rmk:1})  below.

\item \label{rmk:1} In what follows we will be working with three Hopf algebras. We will consider $(\mathcal U(\mathfrak g),\cdot,\Delta)$, i.e., the universal enveloping algebra of $\mathfrak g$, and the Hopf algebra $(\mathcal{U}(\mathfrak g), *, \Delta)$, both def\/ined on the same underlying vector space, whose products are $\mu_\cdot(x,y)=x \cdot y$ and $\mu_\ast(x,y)= x * y$, respectively. Note that the coproduct $\Delta$ is the same for both, given by the one originally def\/ined on the universal enveloping algebra $(\mathcal U(\mathfrak g),\cdot,\Delta)$. According to the last statement in Theorem \ref{thm:OudomGuin} this map is an algebra morphism on $(\mathcal U(\mathfrak g),*)$. The third Hopf algebra we will consider is the universal enveloping algebra of the Lie algebra $(\overline{\mathfrak g},\llbracket \cdot , \cdot \rrbracket )$. Even though this Hopf algebra should be denoted as $(\mathcal U(\overline{\mathfrak g}),\overline{\cdot},\Delta_{\overline{\cdot}})$ we stick to a simplif\/ied notation, and denote its product and coproduct by $\cdot$ and $\Delta$, respectively.

\item \label{rmk:2} From now on, suitable completions of the above Hopf algebras will be considered, still denoted by the same symbols $\mathcal U(\mathfrak g):=(\mathcal U(\mathfrak g),\cdot,\Delta,)$, $\mathcal U_{\ast}(\mathfrak g):=(\mathcal{U}(\mathfrak g), *, \Delta)$ and  $\mathcal U(\overline{\mathfrak g}):=(\mathcal U(\overline{\mathfrak g}),\cdot,\Delta)$. Then, for any element $v \in\mathfrak g$,  one may consider the elements $\exp^{\cdot} v$, $\exp^{\overline{\cdot}} v$ and $\exp^*v$. For example, with $v^{*n}$ denoting the $n$-fold product $v * \cdots * v$,
\begin{gather*}
	\exp^{*} v := \sum_{n\geq 0}\frac{v^{*n}}{n!} \in \mathcal{U}_*(\mathfrak g).
\end{gather*}
A simple computation shows that each of these elements is \emph{group-like} in the corresponding Hopf algebra. For this reason one may consider $G$, $G^*$ and $\overline{G}$, the groups generated for $v \in \mathfrak g$ by the products of the elements of type $\exp^{\cdot} v$, $\exp^{\ast} v$ and $\exp^{\overline{\cdot}}v$, respectively. Moreover, to simplify notation, we will write $\exp v$ and $\exp^{\cdot}v$ to denote the exponentials of $v$ in $\mathcal U(\mathfrak g)$ and $\mathcal U(\bar{\mathfrak g})$, respectively.

\item \label{rmk:3} In what follows we will often need to use the classical $\BCH$-formula. Recall that $\BCH \colon {\mathfrak g} \times {\mathfrak g}\rightarrow {\mathfrak g}$ is def\/ined such that $\exp x \exp y = \exp \BCH(x,y)$, and
\begin{gather*}
	\BCH(x,y) = x + y + \frac{1}{2} [x,y] + \frac{1}{12} \big[x,[x,y]\big]
		 				+ \frac{1}{12} \big[y,[y,x]\big] - \frac{1}{24} \big[y,[x,[x,y]]\big] + \cdots,
\end{gather*}
where $x,y \in \mathfrak{g}$. The \emph{reduced $\BCH$-formula} is def\/ined by $\overline{\BCH}(x,y) := \BCH(x,y) - x - y$. From a formal point of view the $\BCH$-formula maps elements from~$\mathfrak g$ into the completion of~$\mathcal U(\mathfrak g)$. Without further comments, we will therefore assume that it is convergent. How\-ever, when working locally we will restrict ourselves to a suitable neighborhood~$U$ of the zero element of the Lie algebra~$\mathfrak g$.
\end{enumerate}
\end{Remark}

Beside the exponential maps introduced in (\ref{rmk:2}), we need the following \emph{ordered post-Lie exponential} in $\mathcal{U}(\mathfrak g)$.

\begin{Definition}
For any primitive element  $a \in \mathfrak g$, the right-ordered exponential $\exp^{\tr}\colon \mathfrak g \to \mathcal{U}(\mathfrak g)$ is def\/ined for $b \in \mathcal{U}(\mathfrak g)$ as
\begin{gather*}
	\exp^{\tr}(a)b:= b + a \tr b + \frac{1}{2!} a \tr (a \tr b) + \frac{1}{3!} a \tr (a \tr (a \tr b))+ \cdots.
\end{gather*}
\end{Definition}

From the identity $A \tr (B \tr C) = (A_{(1)}(A_{(2)} \tr B)) \tr C$, which holds for all $A, B, C\in\mathcal U(\mathfrak g)$, see~\cite{EFLMK}, it  follows immediately that in $\mathcal{U}(\mathfrak g)$:
\begin{gather*}
	\exp^{\tr}(a)b = \exp^{*}(a) \tr b.
\end{gather*}

We now consider these results in the context of a post-Lie algebra which is def\/ined in terms of an $R$-matrix $\pi_+ \in\operatorname{End}(\mathfrak g)$ satisfying identity~\eqref{mYBE}. First recall that any element $v \in \mathfrak g$ may be decomposed: $v = \pi_-(v) + \pi_+(v) =: v_- + v_+$, and let $(\mathfrak g, [\cdot,\cdot], \tr)$ be the corresponding post-Lie algebra with $a \triangleright b := - [\pi_+(a),b]$.
Then:
\begin{Proposition}
\label{prop:Ad-exp}
For every $v,w \in \mathfrak g$ the following equality
\begin{gather*}
	\exp^*(v) \tr w 	= \exp\big({-}\pi_+(v)\big)w\exp\big(\pi_+(v)\big)
				= \exp\big({-}\operatorname{ad}_{\pi_+(v)}\big)w
\end{gather*}
holds in $\mathcal{U}(\mathfrak g)$.
\end{Proposition}

\begin{proof}
Note that this result implies that $\exp^*(v) \tr w \in \mathfrak g$ for $v,w \in \mathfrak g$.
\begin{gather*}
	\exp^\tr(v)w 	 = w + v \tr w + \frac{1}{2} v \tr (v \tr w) + \cdots
				 = w - [\pi_+(v),w]  + \frac{1}{2} [\pi_+(v),[ \pi_+(v), w]] - \cdots \\
\hphantom{\exp^\tr(v)w}{} =  \exp\big({-}\pi_+(v)\big) w \exp\big(\pi_+(v)\big). \tag*{\qed}
\end{gather*}
\renewcommand{\qed}{}
\end{proof}

The next proposition is a natural factorization statement for the generators of the group $G^*\subset\mathcal U(\mathfrak g)$, see item~(\ref{rmk:2}) in Remark~\ref{rmk:rem}.

\begin{Proposition} \label{prop:star-fact}
For each $v \in {\mathfrak g}$ the following factorization holds:
\begin{gather}
\label{eq:star-fact}
	\exp^*(v) = \exp(v_-)\exp(v_+).
\end{gather}
\end{Proposition}

\begin{proof}
Recall that for $v\in {\mathfrak g}  \hookrightarrow \mathcal{U}({\mathfrak g})$, one has $v \tr v = -[v_+,v_-]$, where $v_\pm := \pi_\pm(v)$, and $v \tr v_+ = 0$. Next we show inductively that $\frac{1}{n!} v^{*n} = \sum\limits_{k+l=n} \frac{1}{k!}\frac{1}{l!} v_-^k v_+^l$. For $n=2$ one f\/inds that:
\begin{gather*}
	\frac{1}{2} v * v 	= \frac{1}{2} vv + \frac{1}{2} v \tr v
				= \frac{1}{2} vv - \frac{1}{2}  [v_+,v_-]
				= \frac{1}{2} v_+v_+ + \frac{1}{2} v_-v_- + v_-v_+.
\end{gather*}
For $n>2$ we have
\begin{gather*}
	 \frac{1}{n!} v^{*n} = \frac{1}{n (n-1)!} v*v^{*n-1}
				     =  \frac{1}{n} v * \sum_{k=0}^{n-1} \frac{1}{k!}\frac{1}{(n-1-k)!} v_-^k v_+^{n-1-k}\\
\hphantom{\frac{1}{n!} v^{*n}}{}
		 = \frac{1}{n}\left( 	\sum_{k=0}^{n-1} \frac{1}{k!} \frac{1}{(n-1-k)!} v v_-^k v_+^{n-1-k} +
						\sum_{k=0}^{n-1} \frac{1}{k!} \frac{1}{(n-1-k)!} v \tr \big(v_-^k v_+^{n-1-k} \big)\right)\\
\hphantom{\frac{1}{n!} v^{*n}}{}
		= \frac{1}{n}\left( 	\sum_{k=0}^{n-1} \frac{1}{k!} \frac{1}{(n-1-k)!} \big(v_-^{k+1} v_+^{n-1-k} + v_+ v_-^{k} v_+^{n-1-k}\big) \right.\\
\left. \hphantom{\frac{1}{n!} v^{*n}=}{}
+
					      	\sum_{k=1}^{n-1} \frac{1}{k!} \frac{1}{(n-1-k)!} \big(v \tr v_-^{k}\big) v_+^{n-1-k}\right),
\end{gather*}
where we used that $v \tr v_+ =0$. Furthermore
\begin{gather*}
	\big(v \tr v_-^{k}\big) = (v \tr v_-)v_-^{k-1} + \sum_{m=1}^{k-1} v_-^m (v \tr v_-)v_-^{k-m-1}\\
\hphantom{\big(v \tr v_-^{k}\big)}{}
			  = -[v_+, v_-]v_-^{k-1} - \sum_{m=1}^{k-1} v_-^m ([v_+, v_-])v_-^{k-m-1}\\
\hphantom{\big(v \tr v_-^{k}\big)}{}
= -v_+v_-^{k} + v_- v_+  v_-^{k-1}  - v_- v_+v_-^{k-1} + v_-^2 v_+ v_-^{k-2}
			  - \sum_{m=2}^{k-1} v_-^m ([v_+, v_-])v_-^{k-m-1}\\
\hphantom{\big(v \tr v_-^{k}\big)}{}
= -v_+v_-^{k} + v_-^{k}v_+,
\end{gather*}
which implies that
\begin{gather*}
	\frac{1}{n!} v^{*n}  =  \frac{1}{n}\left( \sum_{k=0}^{n-1} \frac{1}{k!} \frac{1}{(n-1-k)!}
						\big(v_-^{k+1} v_+^{n-1-k} + v_+ v_-^{k} v_+^{n-1-k}\big) \right.\\
	\left.	\hphantom{\frac{1}{n!} v^{*n}  = }{}
+ \sum_{k=0}^{n-1} \frac{1}{k!} \frac{1}{(n-1-k)!} \big({-}v_+v_-^{k} + v_-^{k}v_+\big) v_+^{n-1-k}\right)\\
\hphantom{\frac{1}{n!} v^{*n} }{}
= \frac{1}{n}\left( \sum_{k=1}^{n} \frac{1}{(k-1)!} \frac{1}{(n-k)!}
						v_-^{k} v_+^{n-k} +
						\sum_{k=0}^{n-1} \frac{1}{k!} \frac{1}{(n-1-k)!} v_-^{k}v_+^{n-k}\right)\\	
\hphantom{\frac{1}{n!} v^{*n} }{}
= \frac{1}{n!}v_-^{n} + \frac{1}{n!}v_+^{n} + \frac{1}{n}\left( \sum_{k=1}^{n-1} \frac{1}{(k-1)!} \frac{1}{(n-k)!}
						v_-^{k} v_+^{n-k} +
						\sum_{k=1}^{n-1} \frac{1}{k!} \frac{1}{(n-1-k)!} v_-^{k}v_+^{n-k}\right)\\
\hphantom{\frac{1}{n!} v^{*n} }{}
= \frac{1}{n!}v_-^{n} + \frac{1}{n!}v_+^{n} + 	
					\sum_{k=1}^{n-1} \frac{1}{k!} \frac{1}{(n-k)!} v_-^{k}v_+^{n-k}.	\tag*{\qed}							
\end{gather*}
\renewcommand{\qed}{}
\end{proof}	

\begin{Remark} \quad
\begin{enumerate}[a)]\itemsep=0pt

\item[a)] Another way to prove equality~\eqref{eq:star-fact} is to show that both sides solve the same initial value problem. For this we take a local point of view by assuming a small enough~$t$ such that $\exp^*(tv)$ and $\exp(tv_-)\exp(tv_+)$ both lie in a suf\/f\/iciently small neighborhood of the unit of the Lie group~$G$ corresponding to the Lie algebra~$\mathfrak g$. Then $y_1(t):=\exp^*(tv)$ and $y_2(t):= \exp(tv_-)\exp(tv_+)$ are solutions of
\begin{gather*}
	\dot{y}(t)=y(t)\big( \exp(-\pi_+(tv)) v \exp(\pi_+(tv))\big), \qquad y(0)=1.
\end{gather*}
Indeed,
\begin{gather*}
	\dot{y}_1(t)  =  y_1(t) * v
	 =  y_1(t) (y_1(t) \tr v )
	= y_1(t)\big(\exp^*(tv) \tr v\big)\\
\hphantom{\dot{y}_1(t)}{}
	= y_1(t)\big( \exp (-\pi_+(tv) ) v \exp (\pi_+(tv) )\big),	
\end{gather*}	
and
\begin{gather*}
	\dot{y}_2(t) 	 =   \exp(tv_-)v_-\exp(tv_+) + \exp(tv_-)v_+\exp(tv_+)\\
\hphantom{\dot{y}_2(t)}{} =  \exp(tv_-)v\exp(tv_+)
				= y_2(t)\big(  \exp (-\pi_+(tv) ) v \exp (\pi_+(tv) ) \big).
\end{gather*}
From which the statement follows by uniqueness of the solution.

\item[b)] Returning to item~\eqref{rmk:3} of Remark~\ref{rmk:rem}, we f\/ind that~\eqref{eq:star-fact} implies
\begin{gather*}
	\exp^*(v) 	= \exp(v_-)\exp(v_+)= \exp\big(\BCH(v_-, v_+)\big).
\end{gather*}
In light of Proposition~3.11 in~\cite{EFLMK}, which says that for $v \in \mathfrak g$ and suf\/f\/iciently small enough $t$
\begin{gather*}
	\exp^*(tv)=\exp (\theta(tv) ),
\end{gather*}
where $\dot{\theta}(tv)={\rm{dexp}}^{-1}_{\theta(tv)}(\exp\big(\theta(tv)\big) \tr v)$, we see that $\theta(tv)= \BCH(tv_-, tv_+)$.
\end{enumerate}
\end{Remark}

If the $R$-matrix corresponds to a Lie algebra that splits into a direct sum of vector subspaces, $\mathfrak g = \mathfrak g_+ \oplus \mathfrak g_-$, then in addition to $v \tr v_+ = 0$, we have that $v_- \tr v_- = v_+ \tr v_+ = 0$. This implies the next result.

\begin{Corollary}
In the case of an $R$-matrix $\pi_+$ which is a projector on~${\mathfrak g}$, we find that
\begin{gather*}
	\exp^*(v_\pm) = \exp(v_\pm).
\end{gather*}
\end{Corollary}

The result in Proposition~\ref{prop:star-fact} has a Hopf algebraic formulation. Recall that the universal property of $\mathcal{U}(\bar{\mathfrak{g}})$ implies that the $R$-matrices ${\rm{r}}_+:= - \pi_+$  and ${\rm{r}}_-:= \pi_-$ become unital algebra morphism from $\mathcal{U}(\bar{\mathfrak{g}})$ to~$\mathcal{U}(\mathfrak g)$.

\begin{Definition}\label{def:HF}
Let $F\colon \mathcal{U}(\bar{\mathfrak{g}}) \to \mathcal{U}(\mathfrak{g})$ be the map def\/ined by
\begin{gather}
\label{Hopf-fact1}
	F:= \mu \circ (\operatorname{id} \otimes S) \circ ({\rm{r}}_- \otimes {\rm{r}}_+)   \circ \Delta,
\end{gather}
where $\mu$ denotes the product in $\mathcal{U}(\mathfrak{g})$.
\end{Definition}

\begin{Remark}
For $X\in\mathcal U(\mathfrak g)$, using the Sweedler notation, one can write
\begin{gather*}
	F(X) = (-1)^{|X_{(2)}|}{X_-}_{(1)}S\big({X_+}_{(2)}\big),
\end{gather*}
where $|X|$ denotes the degree of a homogenous $X \in \mathcal U(\mathfrak g)$.
\end{Remark}

We can now prove the following result.

\begin{Proposition}\label{prop:isoal}
The map $F$ is an algebra morphism from $\mathcal{U}(\bar{\mathfrak{g}})$ to $\mathcal{U}_{*}(\mathfrak{g})$, i.e., it is a linear map such that:
\begin{gather*}
	F(x_1 \cdots   x_n) = F(x_1) * \cdots * F(x_n),
\end{gather*}
for all monomials $x_1 \cdots  x_n\in\mathcal U(\bar{\mathfrak g})$.
\end{Proposition}

\begin{proof}
First observe that since every $x \in \bar{\mathfrak{g}}$ is a primitive element in $\mathcal{U}(\bar{\mathfrak{g}})$,
\begin{gather*}
	F(x) = {\rm{r}}_-(x) - {\rm{r}}_+(x) = x_- + x_+ =x.
\end{gather*}
Then note that for every $x,y \in \mathfrak{g}$ one f\/inds:
\begin{gather*}
	F(x \cdot y) = x_-y_- + y_+x_+ + x_-y_+ + y_-x_+ = xy - [\pi_+(x),y] = x*y=F(x)*F(y).
\end{gather*}
Let $X:=x_1 \cdot x_2 \cdots  x_n = x_1 \cdot X' \in \mathcal{U}(\bar{\mathfrak{g}})$, $x_i \in \mathfrak{g}$, and calculate
\begin{gather*}
	F(X) =  \mu \circ (\operatorname{id} \otimes S) \circ ({\rm{r}}_- \otimes {\rm{r}}_+) \circ \Delta(x_1 \cdot X' )\\
\hphantom{F(X)}{}
=  \mu \circ (\operatorname{id} \otimes S) \circ ({\rm{r}}_- \otimes {\rm{r}}_+) \circ \Delta(x_1) \cdot \Delta_\cdot(X' )\\
\hphantom{F(X)}{}
=  \mu \circ ({x_1}_- \otimes \un) (\operatorname{id} \otimes S)
										\circ ({\rm{r}}_- \otimes {\rm{r}}_+) \circ \Delta(X' )\\
\hphantom{F(X)=}{}
  -\mu \circ (\operatorname{id} \otimes S) \circ (\un \otimes {x_1 }_+)
										 ({\rm{r}}_- \otimes {\rm{r}}_+) \circ \Delta(X' )\\ 							
\hphantom{F(X)}{}
=  \mu \circ ({x_1 }_- \otimes \un) (\operatorname{id} \otimes S)
										\circ ({\rm{r}}_- \otimes {\rm{r}}_+) \circ \Delta(X' )\\
\hphantom{F(X)=}{}  + \mu \circ \big( ((\operatorname{id} \otimes S) \circ  ({\rm{r}}_- \otimes {\rm{r}}_+)
										\circ \Delta(X')  ) (\un \otimes {x_1 }_+)\big)\\
\hphantom{F(X)}{}
= {x_1 }_- (x_2 * \cdots  * x_n) + (x_2 * \cdots  * x_n){x_1 }_+\\
\hphantom{F(X)}{}
= {x_1 }(x_2 * \cdots  * x_n) - {x_1 }_+ (x_2 * \cdots  * x_n) 	+ (x_2 * \cdots  * x_n){x_1}_+\\
\hphantom{F(X)}{}
= x_1*x_2 * \cdots  * x_n = F(x_1)*F(x_2) * \cdots  *F(x_n).	\tag*{\qed}						 		
\end{gather*}	
  \renewcommand{\qed}{}
\end{proof}

We conclude this section by observing that using the map $F$ of Def\/inition~\ref{def:HF}, one can recover the factorization described in Proposition \ref{prop:star-fact}.

\begin{Corollary}
For every $v\in\mathfrak g$,
\begin{gather*}
	F(\exp^\cdot (v)) = \exp(v_-)\exp(v_+)=\exp^*(v).
\end{gather*}
\end{Corollary}
\begin{proof}
The result follows from the def\/inition of the map $F$ and from the property of $\exp^\cdot (v)$ being a group-like element in $\mathcal{U}(\bar{\mathfrak{g}})$.
\end{proof}

\begin{Remark}\label{rem:im}
We compare the map $F$ def\/ined in~(\ref{Hopf-fact1}) with formula~(1.19) in~\cite{RSTS}. The authors of~\cite{RSTS} work with a factorizable $r$-matrix, i.e., with an element $r \in \mathfrak g \otimes \mathfrak g$, satisfying the classical Yang--Baxter equation, and having a symmetric part def\/ining a linear isomorphism $I\colon \mathfrak g^*\rightarrow\mathfrak g$. The element $r$ permits one to def\/ine a Lie algebra struture on the dual vector space~$\mathfrak g^*$. Let~$\mathcal U(\mathfrak g^{\ast})$ be the corresponding universal enveloping algebra. Then in~\cite{RSTS} it is proven that the map
\begin{gather*}
	\mathcal I = \mu \circ (\operatorname{id} \otimes S) \circ (r_+\otimes r_-) \circ \Delta \colon \
	\mathcal U(\mathfrak g^*)\rightarrow\mathcal U(\mathfrak g)
\end{gather*}
is a linear isomorphism extending~$I$.

It is easy to show that the Hopf algebra $\mathcal U(\mathfrak g^{\ast})$ is isomorphic to $\mathcal U(\bar{\mathfrak g})$, where the $R$-matrix def\/ining the Lie algebra structure in $\bar{\mathfrak g}$ is obtained from $r$ and $I$ as $R:=\underline{r}\circ I$. Here $\underline{r}\colon \mathfrak g^{\ast}\rightarrow\mathfrak g$ is def\/ined by $\langle\beta,\underline{r}(\alpha)\rangle=\langle\alpha\otimes\beta,r\rangle$, for all $\alpha,\beta\in\mathfrak g^*$. After identifying $\mathcal U(\mathfrak g^*)$ with $\mathcal U(\bar{\mathfrak g})$ using this isomorphism, the map $\mathcal I$ becomes the map $F$, seen as a linear map between~$\mathcal U(\bar{\mathfrak g})$ and~$\mathcal U(\mathfrak g)$. Proposition~\ref{prop:isoal} states that, at the cost of trading the associative product of $\mathcal U(\mathfrak g)$ for the pro\-duct~$*$ def\/ined in~(\ref{def:preLie}), $F$ becomes an isomorphism of associative algebras.

It is also worth mentioning that in \cite{STS1}, the (inverse) of the linear isomorphism def\/ined in formula~\eqref{Hopf-fact1} was shown to restrict to an algebra homomorphism between the center of $\mathcal{U}(\mathfrak g)$ and $\mathcal{U}(\overline{\mathfrak g})$. This homomorphism was then used in the same paper to def\/ine a quantization of the symmetric algebra $\mathcal S(\overline{\mathfrak g})$, i.e., a linear map $q\colon \mathcal S(\overline{\mathfrak g})\rightarrow\mathcal{U}(\overline{\mathfrak g})$, compatible with the f\/iltrations and preserving the so called total symbol of the elements of~$\mathcal{U}(\overline{\mathfrak g})$, see \cite[p.~3414]{STS1}.
The quantization of~$\mathcal S(\overline{\mathfrak g})$ so obtained was then used to construct the quantum integral of motions for systems with linear Poisson brackets.

Regarding the aforementioned remark, we should emphasize that in this note our main interest lies in applying Magnus type formulas combined with post-Lie structures to solve classical Lax type equations, we will refrain from further commenting on possible applications  of the formalism here introduced in the context of the quantum dynamical systems.
\end{Remark}

\section{Post-Lie algebra and Lie-bracket f\/lows}
\label{sect:PostLieBracketFlow}

In this section the assumptions and notations made explicit in items~\eqref{rmk:1} and~\eqref{rmk:2} of Remark~\ref{rmk:rem} apply\footnote{Regarding item~\eqref{rmk:3} in Remark~\ref{rmk:rem} we note that in this section we have decided to work with formal series in the parameter~$t$. However, we could have taken a local point of view by assuming convergence of the $\BCH$-series together with small enough~$t$, which would imply that the exponential maps $\exp$ and~$\exp^*$ lie in suf\/f\/iciently small neighborhoods of the units of the groups~$G$ and $\bar{G}$, respectively, corresponding to the Lie algebras~$\mathfrak g$ and $\bar{\mathfrak{g}}$.}. We start by recalling the so-called $\BCH$-recursion \cite{EGM}. It is def\/ined for any element $x \in \mathfrak g$ through the recursion
\begin{gather}
\label{BCHrecursion1}
	\chi(t x) :=  t x + \overline{\BCH} (-\pi_+ (\chi(t x)),t x ) \in \mathcal U(\mathfrak g)[[t]].
\end{gather}

\begin{Remark}
Note that $\chi(t x)$ is a formal power series whose coef\/f\/icients are iterated commutators. For this reason it belongs to the vector subspace $\mathfrak g[[t]]\subset\mathcal U(\mathfrak g)[[t]]$, whose elements are formal power series with coef\/f\/icients in~$\mathcal U(\mathfrak g)$ of degree~$1$.
\end{Remark}

The f\/irst few terms of the expansion $\chi(t x) := \sum\limits_{n>0} t^n \chi_n(x)$ are given by
\begin{gather*}
	\chi(t x) = t x - \frac{t^2}{2} [\pi_+(x),x] + \frac{t^3}{4} [\pi_+([\pi_+(x),x]),x]			\\
\hphantom{\chi(t x) =}{}		+  \frac{t^3}{12} \big([\pi_+(x),[\pi_+(x),x]] - [[\pi_+(x),x],x]\big) + \cdots.
\end{gather*}
Fix $a_0 \in \mathfrak g$, and suppose that $\pi_-$ is an $R$-matrix. In~\cite{EGM} it was shown, that~\eqref{BCHrecursion1} implies in~$\mathcal{U}(\mathfrak g)[[t]]$ the factorization
\begin{gather}
\label{chiFact1}
	\exp(t a_0) = \exp (\pi_+(\chi(t a_0)) ) \exp (\pi_-(\chi(t a_0)) ).
\end{gather}
Note that
\begin{gather*}
	\exp(t a_0) =  \exp (-\pi_-(\chi(-t a_0)) )\exp (-\pi_+(\chi(-t a_0)) ).
\end{gather*}

If $\mathfrak g = \mathfrak g_+ \oplus \mathfrak g_-$, and if the corresponding projectors $\pi_\pm\colon \mathfrak g \to \mathfrak g_\pm$ satisfy~\eqref{mYBE}, then the above factorization is unique.

Recall that the binary product $x \triangleright y := - [\pi_+(x),y]$ def\/ines a post-Lie algebra structure on $\mathfrak g$, and
\begin{gather*}
	\llbracket x,y \rrbracket  = [\pi_-(x),y] + [x,\pi_-(y)] - [x,y] = x \tr y - y \tr x + [x,y]
\end{gather*}
is the double Lie bracket~\eqref{doublebracket}.

\begin{Remark}
In what follows we will work simultaneously with $\mathcal{U}(\mathfrak g)[[t]]$ and $\mathcal{U}_{*}(\mathfrak g)[[t]]$, which are the rings of formal power series with coef\/f\/icients in $\mathcal U(\mathfrak g)$ and $\mathcal U_{*}(\mathfrak g)$, respectively. Note that these two rings are endowed with a natural Hopf algebra structure, inherited from $\mathcal U(\mathfrak g)$ and $\mathcal U_{\ast}(\mathfrak g)$. Furthermore, note that the last statement in Theorem~\ref{thm:OudomGuin} implies that $\mathcal U_{*}(\mathfrak g)[[t]]$ is isomorphic, as a Hopf algebra, to $\mathcal U(\bar{\mathfrak g})[[t]]$.
\end{Remark}

From \eqref{chiFact1} and \eqref{eq:star-fact} we deduce the next proposition.

\begin{Proposition}
\label{prop:key}
For any $a \in \mathfrak g$, the following identity holds in  $\mathcal{U}(\mathfrak g)[[t]]$
\begin{gather}
\label{key1}
	\exp(- t a) = \exp^* (-\chi(t a) ).
\end{gather}
\end{Proposition}

\begin{proof}
Statement \eqref{key1} follows from~\eqref{eq:star-fact} by considering $\chi(ta) \in \mathfrak g[[t]]$
\begin{gather*}
	 \exp^* (-\chi(t a) ) =  \exp (-\pi_-(\chi(t a)) ) \exp (-\pi_+(\chi(t a)) ).
\end{gather*}
Then \eqref{chiFact1} implies that $\exp(- t a) = \exp^* (-\chi(t a) )$.
\end{proof}

From \eqref{key1} we see that $\exp(t a) = \exp^* (-\chi (- t a) )$ in  $\mathcal{U}(\mathfrak g)[[t]]$. From Proposition~\ref{prop:key} we get a~dif\/ferential equation for the $\BCH$-recursion~\eqref{BCHrecursion1} in~$\mathcal{U}(\mathfrak g)[[t]]$:
\begin{gather*}
	 {\rm{dexp}}^{*}_{\chi(t a)} (\dot \chi(t a) )
	  			 = \exp^* (\chi(t a) )* (\exp(-t a)a) \\
\hphantom{{\rm{dexp}}^{*}_{\chi(t a)} (\dot \chi(t a) )}{}
=  \exp^* (\chi(t a) )
				 	 (\exp^* (\chi(t a) ) \triangleright   (\exp(-t a)a) ) \\ 
\hphantom{{\rm{dexp}}^{*}_{\chi(t a)} (\dot \chi(t a) )}{}				 				
				   = \exp^* (\chi(t a) )
					 \big(
						(\exp^*(\chi(t a)) \triangleright   \exp(-t a))
						 (\exp^*(\chi(t a)) \triangleright   a)
					 \big)\\
\hphantom{{\rm{dexp}}^{*}_{\chi(t a)} (\dot \chi(t a) )}{}	
				 =	 \exp^*(\chi(t a))
					\big(
						 (\exp^*(\chi(t a)) \triangleright   \exp^*(-\chi(t a)))
				 		(\exp^*(\chi(t a)) \triangleright   a)
					\big)\\
\hphantom{{\rm{dexp}}^{*}_{\chi(t a)} (\dot \chi(t a) )}{}	
				  =
				  	\big( \exp^*(\chi(t a))
				 		(\exp^*(\chi (t a)) \triangleright
						\exp^*(-\chi(t a)))
					\big)
						 (\exp^*(\chi (t a)) \triangleright a)
									\\ 
\hphantom{{\rm{dexp}}^{*}_{\chi(t a)} (\dot \chi(t a) )}{}	
				 =
			 		 \big( \exp^*(\chi(t a)) *
					 \exp^*(-\chi(t a)) \big)
				 	(\exp^*(\chi(t a)) \triangleright   a) 	
				 = \exp^*(\chi(t a)) \triangleright a, 	
\end{gather*}
and therefore
\begin{gather}
\label{proof-key2}
	\dot \chi(ta) =  {\rm dexp}^{*-1}_{\chi(ta)}\big( \exp^*(\chi(t a)) \triangleright   a\big).
\end{gather}

\begin{Theorem}\label{thm:star-sol}
The solution of the differential equation
\begin{gather*}
	\dot{a}(t) = a(t) \tr a(t), \qquad a(0)=a_0
\end{gather*}
in  $\mathcal{U}(\mathfrak g)[[t]]$, is given by
\begin{gather*}
	a(t) :=\exp^* (\chi(a_0t) )\triangleright a_0.
\end{gather*}
\end{Theorem}

\begin{proof}
From \eqref{proof-key2} we f\/ind that
\begin{gather*}
	 \exp^* (\chi(a_0t) )\triangleright a_0 = {\rm{dexp}}^*_{\chi(a_0t)} (\dot {\chi}(a_0t) ) = a(t).
\end{gather*}
This yields
\begin{gather*}
	\dot{a}(t) 	 =  \big({\rm{dexp}}^*_{\chi(a_0t)}(\dot {\chi}(a_0t)) * \exp^*(\chi(a_0t) ) \big)\triangleright a_0\\
\hphantom{\dot{a}(t)}{}
			= {\rm{dexp}}^*_{\chi(a_0t)}(\dot {\chi}(a_0t)) \tr \big(\exp^*(\chi(a_0t)) \triangleright a_0\big)
			= a(t) \tr a(t).\tag*{\qed}
\end{gather*}
  \renewcommand{\qed}{}
\end{proof}

Recall Proposition \ref{prop:Ad-exp}. It implies for $a_0 \in \mathfrak{g}$ that
\begin{gather*}
	a(t) 	= \exp^*(\chi(a_0t))\triangleright a_0 						
		=  \exp(- \pi_+(\chi(a_0t))) a_0 \exp(\pi_+(\chi(a_0t))). 	
\end{gather*}
We deduce that $a(t) \in \mathfrak g[[t]]$. Therefore, assuming that $t$ is suf\/f\/iciently small and the $\BCH$-recursion convergent, such that $\exp^*(\chi(a_0t))$ lies in a neighborhood of the unit of the group~$G$ corresponding to the Lie algebra~$\mathfrak g$, then $a(t) = \exp^*(\chi(a_0t))\triangleright a_0$ solves the Lie bracket equation in~$\mathfrak g${\samepage
\begin{gather*}
	\dot{a}(t) = [a(t),p_+(t)],
\end{gather*}
with initial value $a(0)=a_0$, and $p_+(t)=\pi_+(a(t))$.}

Note that
\begin{gather*}
	{\rm{dexp}}^{*-1}_{x}(y)= \frac{{\rm{ad}}^*_x}{\exp{{\rm{ad}}^*_x} - 1}(y)
		= \sum_{n \ge 0} \frac{B_n}{n!} {\rm{ad}}^{*n}_x(y),
\end{gather*}
where $B_n$ are the Bernoulli numbers, and ${\rm ad}^{*n}_x(y):={{\rm ad}}^{*n-1}_x(\llbracket x,y \rrbracket)$, ${{\rm ad}}^{*0}_x(y):=y$. Hence, with $\chi(t a) := \sum\limits_{n>0} t^n \chi_n(a)$ it follows that
\begin{gather*}
	 {\rm{dexp}}^{*-1}_{\chi(a_0t)}:=\frac{{\rm{ad}}^*_{\chi(a_0t))}}{\exp{{\rm{ad}}^*_{\chi(a_0t)}} -1}
	  = \sum_{s \ge 0} \frac{B_s}{s!} {\rm{ad}}^{*s}_{\chi(a_0t)}\\
\hphantom{{\rm{dexp}}^{*-1}_{\chi(a_0t)}}{}
	  =  1+ \sum_{s > 0} t^s  \sum_{j=1}^{s} \frac{B_j}{j!}  \sum_{k_1 + \cdots +k_j=s \atop k_i>0}
	{\rm{ad}}^*_{\chi_{k_1}} \cdots {\rm{ad}}^*_{\chi_{k_j}}.
\end{gather*}
The post-Lie exponential $\exp^*\big(\chi(a_0t)\big)\triangleright a_0$ is def\/ined through
\begin{gather*}
	\exp^{\tr} (\chi(a_0t) )a_0
	:= a_0 + \sum_{i>0} \frac{1}{i!} \chi(a_0t) \tr  ( \chi(a_0t) \cdots \tr (\chi(a_0t) \tr a_0) )\\
\hphantom{\exp^{\tr} (\chi(a_0t) )a_0}{}
	=  a_0 + \sum_{i>0} \frac{t^i}{i!}  \sum_{u=1}^{i} \sum_{k_1 + \cdots +k_u=i \atop k_i>0}
			{\chi_{k_1}} \tr ({\chi_{k_2}} \tr \cdots ({\chi_{k_u}} \tr a_0)),
\end{gather*}
which yields for $\Omega_\tr(a_0t) :=  \sum\limits_{n>0} t^n \Omega_n(a_0) = \int_0^t  \dot {\chi}(a_0s)ds$
\begin{gather*}
	\Omega_1 = a_0, \qquad \Omega_2 = \frac{1}{2} a_0 \tr a_0,
\end{gather*}
and for $n>2$
\begin{gather*}
	 n\Omega_n = \sum_{j=1}^{n-1} \frac{1}{j!} \!\sum_{k_1 + \cdots +k_j=n-1 \atop k_i>0}\!
					{\Omega_{k_1}} \tr ({\Omega_{k_2}} \tr \cdots ({\Omega_{k_j}} \tr a_0) )
					+  \sum_{j=1}^{n-1} \frac{B_j}{j!} \!\sum_{k_1 + \cdots +k_j=n-1 \atop k_i>0}\!
					{\rm ad}^*_{\Omega_{k_1}} \cdots {\rm ad}^*_{\Omega_{k_j}}a_0 \\
\hphantom{n\Omega_n =}{}
			     + \sum_{j=2}^{n-1}\left( \left(  \sum_{s=1}^{j-1} \frac{B_s}{s!}
			    		\sum_{k_1 + \cdots +k_s=j-1 \atop k_i>0}
					{\rm ad}^*_{\Omega_{k_1}} \cdots {\rm ad}^*_{\Omega_{k_s}}\right)\right.\\
\left.\hphantom{n\Omega_n =}{}\times
					\left( \sum_{u=1}^{n-j} \frac{1}{u!} \sum_{k_1 + \cdots +k_u=n-j \atop k_i>0}
					{\Omega_{k_1}} \tr ({\Omega_{k_2}} \tr \cdots ({\Omega_{k_u}} \tr a_0) )\right)\right).
\end{gather*}

\begin{Remark}
This formula is equivalent to equation~(14) in~\cite{CasasIserles} (see also~\cite{Casas}), using $[\pi_+(x),\pi_+(y)]$ $= -\pi_+(\llbracket x,y \rrbracket )$, such that
\begin{gather*}
	\pi_+\big({\rm ad}^*_{\Omega_{k_1}} \cdots {\rm ad}^*_{\Omega_{k_u}}a_0\big) = (-1)^u {\rm ad}_{\Omega_{k_1}} \cdots {\rm ad}_{\Omega_{k_u}}\pi_+(a_0).
\end{gather*}
\end{Remark}

\begin{Theorem}
\label{thm:star-product}
Given $x \in \mathfrak g$, and $\xi \in \mathcal U(\mathfrak g)$. Then in $\mathcal U(\mathfrak g)[[t]]$
\begin{gather*}
	\exp(tx) * \xi = \exp (-\pi_- (\chi (-tx) )  ) \xi  \exp (-\pi_+ (\chi (-tx) ) ) .
\end{gather*}
\end{Theorem}

\begin{proof}
First we calculate $\exp(tx) * \xi $ for $\xi \in \mathfrak g$. We f\/ind
\begin{gather*}
	\exp(tx) * \xi =  \exp(tx)(\exp(tx) \triangleright   \xi )
			= \exp(tx)(\exp^*(-\chi (- tx)) \triangleright   \xi),
\end{gather*}
where we used \eqref{key1}. Then
\begin{gather*}
	\exp^*(-\chi(- tx)) \triangleright   \xi
		= \exp^{\triangleright  }(-\chi (- tx))\xi
		= \exp(\pi_+(\chi (-tx))) \xi \exp(- \pi_+(\chi (-tx)) ).
\end{gather*}
Therefore
\begin{gather*}
 \exp(tx) * \xi =  \exp(tx) (\exp^* (-\chi (- tx) ) \triangleright   \xi )\\
 \hphantom{\exp(tx) * \xi }{}
			 =  \exp (-\pi_-(\chi (-tx))) \exp(-\pi_+(\chi (-tx)) )
					\exp(\pi_+(\chi (-tx)) ) \xi \exp(- \pi_+(\chi (-tx)) )\\
\hphantom{\exp(tx) * \xi }{}
			= \exp(-\pi_-(\chi (-tx))) \xi \exp(- \pi_+(\chi (-tx)) ) .
\end{gather*}

Now we suppose that $\xi \in \mathcal U(\mathfrak g)$. For notational simplicity we assume that $\xi=a^n$, $a \in \mathfrak g$.
\begin{gather*}
	\exp^*(-\chi (- tx)) \triangleright   a^n
		= \exp^{\triangleright  }(\chi (-t x)) a^n\\
\hphantom{\exp^*(-\chi (- tx)) \triangleright   a^n}{}
		= a^n- \chi (- tx) \triangleright  a^n
				+ \frac{1}{2!}  \chi (-t x) \tr (\chi (- tx) \triangleright a^n ) + \cdots.
\end{gather*}
Recall that for $\chi (- tx) \in \mathfrak g[[t]]$
\begin{gather*}
	\chi(- tx) \triangleright  a^n = \sum_{i=1}^n a^{i-1}(\chi (- tx) \triangleright  a)a^{n-i}
	= - \sum_{i=1}^n a^i[\pi_+(\chi(- tx)) , a]a^{n-i} \\
\hphantom{\chi(- tx) \triangleright  a^n}{}
	=  - [\pi_+(\chi(- tx)) , a^n].		
\end{gather*}
The last Lie bracket belongs to $\mathcal U(\mathfrak g)[[t]]$. This yields
\begin{gather*}
		\chi(- tx) \triangleright  a^n= \exp(\pi_+(\chi(-tx)) ) a^n \exp(- \pi_+(\chi(-tx)) ),
\end{gather*}
which implies the result.
\end{proof}

\begin{Theorem}
\label{thm:group-star}
For every $x,y\in\mathfrak g$,
\begin{gather}
	\exp(tx)*\exp(y)= \exp (- \pi_- (\chi(-tx)) ) \exp(y) \exp(- \pi_+(\chi(-tx)) ),\label{eq:prod}
\end{gather}
and
\begin{gather*}
	(\exp(tx))^{*-1} =   \exp(\pi_-(\chi(-tx)) ) \exp(\pi_+(\chi(-tx)) ). 
\end{gather*}
\end{Theorem}

\begin{proof}
Recall that we write $\exp v$, $\exp^* v$ and $\exp^{\cdot}v$ to denote the exponentials of~$v$ in $\mathcal U(\mathfrak g)$, $\mathcal U_*(\mathfrak g)$ and $\mathcal U(\bar{\mathfrak g})$, respectively.

Identity \eqref{eq:prod} follows from~\eqref{key1}, and applying the map~$F$ to
\begin{gather*}
\exp^\cdot  (-\chi(- tx) )\cdot \exp^\cdot  (-\chi(- tx) ) \in \mathcal U(\bar{\mathfrak{g}})[[t]].
\end{gather*}
The morphism property of $F$ yields
\begin{gather*}
 F\big(\exp^\cdot  (-\chi(- tx) )\cdot \exp^\cdot  (-\chi(- tx) )\big)\\
 \qquad{} =  \exp^* (-\chi(-t x) )* \exp^*  (-\chi(- tx) ) =   \exp(tx)*\exp(y),
\end{gather*}
and
\begin{gather*}		
 F\big(\exp^\cdot (-\chi(- tx))\cdot \exp^\cdot (-\chi(-t x))\big)\\
\qquad{} = \exp (-\pi_-(\chi(- tx)))  \exp (-\pi_-(\chi(- ty))) \exp (-\pi_+(\chi(- ty))) \exp (-\pi_+(\chi(- tx)))\\
\qquad{} = \exp (-\pi_-(\chi(-t x))) \exp(y) \exp (-\pi_+(\chi(- tx))).\tag*{\qed}
\end{gather*}
\renewcommand{\qed}{}
\end{proof}

\section{Conclusions}
\label{sect:conclusion}

 We have addressed the problem of f\/inding explicit solutions to isospectral f\/low equations. These form a class of dif\/ferential equations usually encountered in the theory of (Hamiltonian) dyna\-mical systems, and commonly studied using methods borrowed from Lie theory. In general, their solutions are obtained starting from the existence of a solution of the modif\/ied classical Yang--Baxter equation. In other words, from the existence of a classical $R$-matrix on the underlying Lie algebra.

 We have outlined how to approach the problem using a framework based on a particular class of non-associative algebras known as post-Lie algebras. More precisely, starting from a classical $R$-matrix solving the modif\/ied classical Yang--Baxter equation on a (f\/inite-dimensional) Lie algebra $\mathfrak g$, we showed how the corresponding post-Lie algebra structure on the Lie enveloping algebra $\mathcal U(\mathfrak g)$ allows us to describe explicit solutions of exponential type to any isospectral f\/low equation def\/ined on the original Lie algebra~$\mathfrak g$.

\subsection*{Acknowledgements}
We acknowledge support from~{ICMAT} and the Severo Ochoa Excelencia Program, as well as the NILS-Abel project 010-ABEL-CM-2014A, and the SPIRIT project RCN:231632 of the Norwegian Research Council. The f\/irst author is supported by a Ram\'on y Cajal research grant from the Spanish government, and acknowledges support from the Spanish government under project MTM2013-46553-C3-2-P.

\pdfbookmark[1]{References}{ref}
\LastPageEnding

\end{document}